\documentclass[11pt, reqno]{amsart}
 
\usepackage{amsmath,amsfonts,amsthm,amssymb,color,hyperref, mathrsfs, graphicx}

\usepackage[T1]{fontenc}

\usepackage[left=1.1in, right=1.1in, top=1.1in,bottom=1.1in]{geometry} 

\newtheorem{theorem}{Theorem}[section]
\newtheorem{lemma}[theorem]{Lemma}
\newtheorem{proposition}[theorem]{Proposition}

\theoremstyle{definition}

\theoremstyle{remark}

\numberwithin{equation}{section}

\newcommand{\R}{\mathbb R}
\newcommand{\RR}{\mathbb R}

\newcommand{\E}{\mathbb E}
 \newcommand{\FF}{\mathscr{F}}
 \newcommand{\cF}{\mathcal{F}}
\newcommand{\HH}{\mathfrak{H}}
 \newcommand{\PP}{\mathbb P}

\newcommand{\LL}{\pmb \ell}
\newcommand{\1}{\textbf{1}}

 \newcommand{\wt}{\widetilde}
 \newcommand{\wh}{\widehat}

\begin{document}

\title[CLTs for stochastic wave equations]{Central limit theorems  for  stochastic wave equations in dimensions one and two}

 \date{\today}

\author[D. Nualart]{David Nualart$^{\ast,1}$}

\author[G. Zheng]{Guangqu Zheng$^{\ast,2}$}

\maketitle

\vspace{-0.5cm}

\begin{center}  {\small \textit{University of Kansas}$^\ast$; nualart@ku.edu$^1$, zhengguangqu@gmail.com$^2$ } \end{center}

\begin{abstract} Fix $d\in\{1,2\}$,
we consider a $d$-dimensional stochastic wave  equation driven by a   Gaussian    noise, which is temporally white and  colored in space such that  the spatial correlation function is integrable and satisfies   Dalang's condition. In this setting, we provide   quantitative central limit theorems for the spatial average of the solution over a Euclidean ball, as the radius of the ball diverges to infinity.  We also establish  functional central limit theorems.  A fundamental ingredient in our analysis is the pointwise $L^p$-estimate for the Malliavin derivative of the solution, which is of independent interest.  This paper is another addendum to the recent research line of averaging stochastic partial differential equations.

  \end{abstract}



\allowdisplaybreaks

\section{Introduction}
In this article, we fix $d\in\{1,2\}$ and
consider the  stochastic wave equation  
\begin{equation}  \label{2dSWE}
\dfrac{\partial^2 u}{\partial t^2} = \Delta u + \sigma(u)   \dot{W},     \\
\end{equation}
on $\R_+\times \R^d$ with initial conditions 
$ u(0,x)=1$ and $ \frac {\partial  u} {\partial t} (0,x)=0 $,  where $\Delta$ is Laplacian in space variables and $\dot{W}$ is a centered Gaussian  noise with
covariance
\begin{equation}
\E[ \dot{W}(t,x)\dot{W}(s,y)    ] =  \delta_0(t-s) \gamma( x-y).    \label{cov}
\end{equation}
Here  $ \dot{W}$  is a distribution-valued field  and will be formally introduced in Section \ref{sub22}.

\medskip 

\emph{Throughout this article}, we fix the following conditions:
\begin{itemize}
 \item[\textbf{(C1)}]  $\sigma:\R\to\R$ is Lipschitz continuous with Lipschitz constant $L\in(0,\infty)$.
 
 \item[\textbf{(C2)}] $\gamma$ is a tempered  nonnegative and nonnegative definite function, whose Fourier transform  $\mu$  satisfies  Dalang's condition:
  \begin{align}\label{DalangC}
  \int_{\R^d} \frac{\mu(dz )}{1+  | z| ^2} < \infty,
 \end{align}
 where $ | \cdot | $ denotes the Euclidean norm on $\R^d$.
 \item[\textbf{(C3)}]   $\sigma(1) \neq 0$.
\end{itemize}
Conditions  \textbf{(C1)} and  \textbf{(C2)}   ensure that    equation \eqref{2dSWE} has a unique  \emph{random field solution}, which is adapted to the filtration generated by $W$, such that $\sup\big\{  \E \big[  \vert u(t,x)\vert^k\big]:   (t,x) \in [0,T]\times \R^d\big\}      $ is finite for all  $T\in(0,\infty)$ and  $k\ge 2$,  and  
\begin{equation}\label{mild}
 u(t,x) = 1+    \int_0^t \int_{\R^d}  G_{t-s}(x-y) \sigma(u(s,y))W(ds,dy), 
\end{equation}
where   the above stochastic integral is defined in the sense of Dalang-Walsh and $G_{t-s}(x-y)$ denotes the fundamental solution to the corresponding deterministic  wave equation, \emph{i.e.} 
\begin{align}\label{fundamentalS}
G_t(x) :=\begin{cases}
  \dfrac{1}{2} \1_{\{ | x| < t\}} ,   &\text{if $d=1$} \\
 \dfrac{1}{2\pi \sqrt{t^2 - | x|^2}} \1_{\{| x | <t \}},  &\text{if $d=2$};
\end{cases}
\end{align}
see   \cite{Dalang99, Dalang}.   Condition  \textbf{(C3)}  excludes the trivial case where $u(t,x) =1$, see Section \ref{Picard}.
 
 \medskip
 We are interested in the behavior of the solution to  \eqref{2dSWE} in the space variable, and the next result provides relevant stationarity and ergodicity properties.
 
 \begin{proposition} \label{prop1}  Suppose that  $\gamma$  satisfies  $\gamma \in L^1(\R)$ if $d=1$ and $\gamma \in L^1(\R^2) \cap L^{\LL}(\R^2)$ for some $\LL >1$ if  $d=2$.
 Fix $t>0$. Then, the random field $\{ u(t,x), x\in \R^d\}$ have the following properties: 
 \begin{itemize}
 
\item[(i)] it is strictly stationary:  The finite-dimensional distributions of  $\{ u(t,x+y), x\in \R^d\}$ does not depend on $y\in\R^d$;

\item[(ii)]  it is ergodic.

\end{itemize}
We postpone the proof to Section \ref{SEC3}.
 \end{proposition}

 We define for each $t\in[0,\infty)$,
\begin{equation} \label{F_R}
F_R(t) =  \int_ {B_R}  (  u(t,x) -1 )~dx,
\end{equation} 
where $B_R= \{ x\in\R^d: | x| \leq R\}$. 
 From   Proposition \ref{prop1}  it follows, applying the Ergodic Theorem, that 
\begin{align}\label{conseq:Erg}
F_R(t) \xrightarrow{R\to\infty} 0 ~\text{almost surely and  in $L^p(\Omega)$ for any $p\geq 1$.}
 \end{align}
 Thus, it is a natural problem to investigate the corresponding central limit theorem (CLT).
 
     We denote the standard Gaussian distribution  by $\mathcal{N}(0, 1)$ and the   $L^p(\Omega)$-norm by $\|\cdot\|_p$ for any $p\ge 1$.   Also,  $\omega_d$ denotes the volume of the unit ball, that is, $\omega_d= 2$ for $d=1$ and $\omega_d= \pi$ for $d=2$.
   We put $a\lesssim b$ if  $a\le Cb$ for some constant $C>0$.

   \medskip
   
In what follows, we present the main result of this article. 
 
   \begin{theorem}\label{MAIN}  
      Suppose that  $\gamma$  satisfies  $\gamma \in L^1(\R)$ if $d=1$ and $\gamma \in L^1(\R^2) \cap L^{\LL}(\R^2)$ for some $\LL >1$ if  $d=2$.
       Then the following statements hold:
      \begin{itemize}
      \item[(i)]  The  process $\big\{ R^{-d/2} F_R(t): t\geq 0\big\}$ converges in law to a centered continuous Gaussian process $\big\{\mathcal{G}_t: t\geq 0\big\}$, where 
   \[
   \E\big[ \mathcal{G}_{t_1}  \mathcal{G}_{t_2} \big]   = \omega_d    \int_{\R^d} \text{\rm Cov} ( u(t_1,\xi),  u(t_2,0) )  d\xi.
   \]
      \item[(ii)]  For any fixed $t>0$,  
    \begin{equation}\label{QCLT}
    d_{\rm TV}\big(  F_R(t)/\sigma_R, \mathcal{N}(0, 1)   \big) \lesssim R^{-d/2},
    \end{equation}
  where  $\sigma_R^2 :=  {\rm Var}(F_R(t))   > 0$ for every $R>0$ is part of the conclusion and $d_{\rm TV}$ stands for the total variation distance.  
  \end{itemize}
   \end{theorem}

In the sequel, we   sketch the ``\textit{usual proof strategy}'' and highlight the key ingredients.  The proof of  the functional CLT consists in proving the \emph{f.d.d.}   convergence and     tightness. We appeal to the tightness criterion of Komogorov-Chentsov (see \emph{e.g.} \cite{Kunita86})  and prove  tightness by obtaining moment estimate of the increments $F_R(t)-F_R(s)$. For the \emph{f.d.d.} convergence, we first derive the asymptotic variance and then apply the so-called Malliavin-Stein approach to show the \emph{f.d.d.} convergence. More precisely, we need a multivariate Malliavin-Stein bound for this purpose, while the univariate Malliavin-Stein bound   provides the rate   for the marginal convergence that is described by the total-variation distance.  It is worth remarking that as a tailor-made combination of Malliavin calculus and Stein's method initiated by Nourdin and Peccati, the Malliavin-Stein approach has proved to be a very useful toolkit in establishing Gaussian fluctuations in various frameworks, notably for the functionals over a Gaussian field, see the recent monograph \cite{NP} by the tailors for a comprehensive treatment.

That being said,   we will use the Malliavin calculus intensively for our computations and inevitably, we will encounter random variables of the form $D_{s,y}u(t,x)$. Note that $Du(t,x)$ denotes the Malliavin derivative of $u(t,x)$, which lives in the Hilbert space $\HH$ associated to the noise $W$; see section \ref{SEC2} for precisely definitions. The space $\HH$ may contain generalized functions, so to estimate $L^p$-norm of $D_{s,y}u(t,x)$, we shall first clarify that $D_{s,y}u(t,x)$ is indeed a real function in $(s,y)$. Moreover, we need to prove  an estimate of the form
\[
\|D_{s,y}u(t,x)\|_p \lesssim G_{t-s}(x-y)
\]
in order to proceed    with our computations for asymptotic variance and  \emph{f.d.d.} convergence.  This is the contents of the following theorem.

        \begin{theorem}\label{THM}  Let the   assumptions   in Theorem \ref{MAIN} prevail.    For any $p\in[2,\infty)$ and any $t>0$, the following estimates hold  for almost all $(s,y) \in [0,t] \times \R^d:$        
 \begin{align}
 G_{t-s}(x-y) \| \sigma(u(s,y))\|_p \leq   \big\| D_{s,y} u(t,x) \big\|_p  \leq C_{p,t,L,\gamma} \kappa_{p,t,L} G_{t-s}(x-y), \label{IMP}
 \end{align} 
 where the constant  $ \kappa_{p,t,L}$ is defined in \eqref{sigmap} and the constant  $C_{p,t,L,\gamma}$ is given by \eqref{def:CPTL} in  the 1D case and by  \eqref{def:CPTL2} in  the 2D case.  
    \end{theorem}
    
    Let us compare our result with similar   estimates in the literature.  There have been several recent works on the application of the Malliavin-Stein  approach to establish central limit theorems for spatial  averages of stochastic partial differential equations and to derive quantitative error bound in the total variation distance.
  A fundamental ingredient  in all these papers is an  upper bound similar to \eqref{IMP}. The works   \cite{HNV18}  and  \cite{HNVZ19} deal with the stochastic heat equation
  with $d=1$ and $\gamma=\delta_0$ and $d \ge 1$ and  $\gamma(x)= |x|^{-\beta}$, $0<\beta < \min(2,d)$ (Riesz kernel covariance), respectively. 
For the stochastic heat equation, an upper bound of the form
    \eqref{IMP},  holds with $G$ being the heat kernel. In this case, the proof   relies heavily on the semigroup property of the heat kernel.
    
    For the wave equation,   the works
 \cite{DNZ18, BNZ20}  establish the Gaussian fluctuation of spatial averages of stochastic wave equations in the  following cases:
 $d=1$ and   $\gamma(x)= |x|^{2H-2}$,  $H\in [1/2,1)$,  and  $d=2$ and $\gamma(x)= |x|^{-\beta}$, $0<\beta < 2$, respectively. In the case $d=1$,  the proof of
 \eqref{IMP} is not very difficult because $G_{t-s}(x-y)$ is uniformly bounded. The case $d=2$ is much more difficult  due to the singularity within the fundamental solution \eqref{fundamentalS}. In the present article, we consider the integrable covariance kernel that requires novel technical estimates and as we can read from Theorem \ref{MAIN}, the order of fluctuation in this case is $R^{d/2}$, which is the same as in the case of parabolic Anderson model driven by integrable covariance kernel \cite{NZ19}.   Our paper can be viewed as another pixel, along with \cite{BNZ20, CKNP19-2, DNZ18, HNV18, HNVZ19, NZ19},   for completing the picture of averaging SPDEs.  It is worth pointing out that the authors of   \cite{BQS2019} considered the 1D linear stochastic wave equation driven by space-time homogeneous Gaussian noise and they obtained a weaker result than \eqref{IMP}. Their methodology is totally different than ours: Due to the linearity, one has the explicit chaos expansion of the solution, then obtaining the upper bound for $\| D_{s,y}u(t,x)\|_p$ reduces to explicit (but very complicated) computations.   And in view of this reference, we believe our bounds in \eqref{IMP} could be very useful in establishing absolute continuity result for the solution to 2D stochastic wave equation.

The rest of this article is organized as follows: In Section \ref{SEC2} we present preliminary results for our proofs, and Sections \ref{SEC3} \ref{SEC:Proof11}, \ref{SEC:Proof12} are devoted to the proofs of  Proposition \ref{prop1}, Theorem \ref{MAIN} and Theorem \ref{THM} respectively.

    \section{Preliminaries} \label{SEC2}
 
   In this section we present some preliminaries on stochastic analysis, Malliavin calculus and the Stein-Malliavin approach to normal approximations.
  
     \subsection{Basic stochastic analysis} \label{sub22}

Let  $\HH$ be defined  as the completion of $C_c(\R_+\times\R^d)$ under the inner product 
\[
\langle f, g\rangle_\HH = \int_{\R_+\times\R^{2d}} f(s,y) g(s, z) \gamma(y-z)dydzds =\int_{\R_+} \left(\int_{\R^d} \FF f(s,\xi) \FF g(s, -\xi) \mu(d\xi)\right)ds, 
\]
where $\FF f(s,\xi)=\int_{\R^d} e^{-i x\cdot \xi} f(s,x)dx$.

Consider  an isonormal Gaussian process associated to the Hilbert space $\HH$, denoted by $W=\big\{ W(\phi): \phi\in \HH \big\}$. That is, $W$ is a centered  Gaussian family of random variables such that
$
\E\big[ W(\phi) W(\psi) \big] = \langle \phi, \psi\rangle_\HH$ for any   $\phi, \psi\in \HH$.
As the noise $W$ is white in time, a martingale structure naturally appears. First we define $\cF_t$ to be the $\sigma$-algebra generated by $\PP$-negligible  sets and $\big\{ W(\phi): \phi\in C(\R_+\times\R^d)$ has compact support contained in $[0,t]\times\R^d \big\}$, so we have a filtration $\mathbb{F}=\{\cF_t: t\in\R_+\}$. If $\big\{\Phi(s,y): (s,y)\in\R_+\times\R^d\big\}$ is  an $\mathbb{F}$-adapted  random field such that $\E\big[ \| \Phi\|_\HH^2 \big] <+\infty$, then 
\[
M_t = \int_{[0,t]\times\R^d} \Phi(s,y)W(ds,dy),
\]
interpreted as the Dalang-Walsh integral (\cite{Dalang99,Walsh}), is a square-integrable $\mathbb{F}$-martingale with quadratic variation  
\[
\langle M \rangle_t = \int_{[0,t]\times\R^{2d}} \Phi(s,y) \Phi(s,z) \gamma(y-z) dy dzds.
\]
Let us record a useful version of Burkholder-Davis-Gundy inequality (BDG for short); see  {\it e.g.} \cite[Theorem B.1]{Khoshnevisan}.

\begin{lemma} \label{BDG}  If $\big\{\Phi(s,y): (s,y)\in\R_+\times\R^d\big\}$ is an adapted random field with respect to $\mathbb{F}$ such that $ \| \Phi\|_\HH \in L^p(\Omega)$ for some $p\geq 2$, then
\[
\left\|  \int_{[0,t]\times\R^d} \Phi(s,y)W(ds,dy) \right\|_p^2 \leq 4p \left\| \int_{[0,t]\times\R^{2d}} \Phi(s,y)\Phi(s,z) \gamma(y-z) dydzds \right\|_{p/2}.
\]
\end{lemma}
\medskip

Now let us recall some basic facts on Malliavin calculus associated with $W$. For any unexplained notation and result,   we refer to the book \cite{Nualart}.  We denote by $C_p^{\infty}(\R^n)$ the space of smooth functions with all their partial derivatives having at most polynomial growth at infinity. Let $\mathcal{S}$ be the space of simple functionals of the form 
$F = f(W(h_1), \dots, W(h_n))
$ for $f\in C_p^{\infty}(\RR^n)$ and $h_i \in \HH$, $1\leq i \leq n$. Then, the Malliavin derivative  $DF$ is the $\HH$-valued random variable given by
\begin{align*}
DF=\sum_{i=1}^n  \frac {\partial f} {\partial x_i} (W(h_1), \dots, W(h_n)) h_i\,.
\end{align*}
 The derivative operator $D$  is   closable   from $L^p(\Omega)$ into $L^p(\Omega;  \HH)$ for any $p \geq1$ and   we define $\mathbb{D}^{1,p}$ to be the completion of $\mathcal{S}$ under the norm
$$
\|F\|_{1,p} = \left(\E\big[ |F|^p \big] +   \E\big[  \|D F\|^p_\HH \big]   \right)^{1/p} \,.
$$
   The {\it chain rule} for $D$ asserts that if $F\in\mathbb{D}^{1,2}$ and $h:\R\to\R$ is Lipschitz, then $h(F)\in\mathbb{D}^{1,2}$  with 
\begin{equation} \label{cr}
D[h(F)] = Y DF,
\end{equation}
where  $Y$ is some  $\sigma\{ F\}$-measurable random variable  bounded by the Lipschitz constant of $h$; when $h$ is  additionally differentiable, we have $Y= h'(F)$.

We denote by $\delta$ the adjoint of  $D$ given by the duality formula 
\begin{equation} \label{D:delta}
\E[\delta(u) F] = \E[ \langle u, DF \rangle_\mathfrak{H}]
\end{equation}
for any $F \in \mathbb{D}^{1,2}$ and $u\in{\rm Dom} \, \delta \subset L^2(\Omega; \HH)$,  the domain of $\delta$. The operator $\delta$ is
also called the Skorohod integral and in the case of the Brownian motion, it coincides
with an extension of the It\^o integral introduced by Skorohod (see \emph{e.g.} \cite{GT, NuPa}).   

In our context, the Dalang-Walsh integral coincides with the Skorohod integral: Any   adapted random field $\Phi$ that satisfies $\E\big[ \|\Phi\|_\HH^2 \big]<\infty $ belongs to the domain of $\delta$  and
\[
\delta (\Phi) = 
\int_0^\infty \int_{\R^d} \Phi(s,y) W(d s, d y).
\]
As a consequence,  the mild formulation equation   \eqref{mild} can    be written as 
\begin{equation*}
u(t,x) = 1 +   \delta\big( G_{t-\bullet}(x-\ast) \sigma (  u(\bullet, \ast)  )\big).
\end{equation*}
      
      With the help of the derivative operator, we can represent $F\in\mathbb{D}^{1,2}$ as a stochastic integral. This is the content of the following two-parameter Clark-Ocone formula, see \emph{e.g.} \cite[Proposition 6.3]{CKNP19} for a proof.
 \begin{lemma}[Clark-Ocone formula] \label{CO}    Given $F\in\mathbb{D}^{1,2}$, we have almost surely
 \[
 F = \E[ F] + \int_{\R_+\times\R^d} \E\big[ D_{s,y} F \vert \mathscr{F}_s \big] W(ds,dy).
 \]
 \end{lemma}
 As a consequence of the above Clark-Ocone formula, we can derive
 the following  Poincar\'e inequality:
For any two random variables $F,G \in \mathbb{D}^{1,2}$, we have
\begin{equation} \label{Poincare}
|{\rm Cov}(F,G )| \le  \int_0^\infty \int_{\R^{2d}}   \|  D_{s,y}F\|_{2}\|  D_{s,z}G\|_{2}  \gamma(y-z) dy dz ds.
\end{equation}

Recall that the total variation distance between two  probability measures
$\mu$ and $\nu$ on $\R$ is defined by
\[
	d_{\rm TV}  (\mu ,\nu)= \sup_{B\in \mathcal{B}(\R)} | \mu(B)- \nu(B)|,
\]
where  $\mathcal{B}(\R)$ denotes the family of all Borel subsets on $\R$. As usual,
 $d_{\rm TV}(F\,,G)$ will  denote the total variation distance between the distribution measures of  $F$ and $G$; and  $d_{\rm TV}(F\,,\mathcal{N}(0\,,1))$ is
 is the total variation distance between $F$ and  a standard Gaussian random variable. 

The combination of Stein's method for normal approximations with Malliavin calculus leads to the following
bound on the total variation distance. See   \cite[Theorem 8.2.1]{Eulalia} for more details.

\begin{proposition}\label{pr:TV}
	Suppose that   $F=\delta(v) \in \mathbb{D}^{1,2}$  has unit variance   for some  $v\in{\rm dom}(\delta)$.   Then
	\begin{equation} \label{SM1}
		d_{\rm TV} (F\,,  \mathcal{ N}(0,1)) \le 2  
		\sqrt{ {\rm Var} \langle DF, v \rangle_{\HH}}.
	\end{equation}	
\end{proposition}

\subsection{Basic formulas}
  We close this section with
 some basic relations for the fundamental solution $G_t(x)$.
 For $t>0$ and  $m>0$, we have 
 \begin{align*}  
 \int_{\R^{2d}}  G_t(y) G_t(z)\gamma(y-z) dydz &=   \int_{\R^d}  \frac{ \sin^2\big(  t | \xi| \big)}{|\xi|^2}  \mu(d\xi)  \\
&=t^2 \int_{| \xi | \leq m}  \frac{ \sin^2\big(  t |\xi | \big)}{t^2|\xi|^2}  \mu(d\xi) + \int_{| \xi | >m}  \frac{ \sin^2\big(  t |\xi | \big)}{ | \xi |^2}  \mu(d\xi) \\
&\leq t^2 \mu (  |\xi| \leq m) +  \int_{ | \xi | >m}  \frac{ \mu(d\xi)}{| \xi |^2}.
 \end{align*}
As a consequence,
   \begin{equation}  \label{FA1}
 \int_{\R^{2d}}  G_t(y) G_t(z)\gamma(y-z) dydz \leq \inf _{m>0}  \left( t^2 \mu ( |\xi| \leq m ) +  \int_{| \xi| >m} |\xi |^{-2} \mu(d\xi)  \right) =: \mathfrak{m}_t.
\end{equation}
It is clear that $\mathfrak{m}_t$ is nondecreasing in $t$.

For any $s<r<t$ and $d=1$, we have
\begin{equation} \label{G1}
 \int_{\R} G_{t-r}(x-z)  G_{r - s}(z-y)  dz \le \frac 12 (t-s) G_{t-s}(x-y).
 \end{equation}

For any $t\in(0,\infty)$, we define 
\begin{equation} \label{varphi}
\varphi_{t,R}(s,y) = \int_{B_R} G_{t-s}(x-y) dx
\end{equation}
where we recall thart $B_R= \{ x \in \R^d: |x| \le R\}$.
In the following lemma, we provide  a useful  estimate about $y\in\R^d\longmapsto\varphi_{t,R}(s,y)$.

  \begin{lemma}\label{lemfact} For $t_1,t_2\in(0,\infty)$,  the quantity 
    \[
  R^{-d}  \int_{\R^{2d}} \varphi_{t_1,R}(s,y) \varphi_{t_2,R}(s,z) \gamma(y-z)   dydz
  \]
  is uniformly bounded over $s\in[0, t_2\wedge t_1]$ and $R>0$, provided $\gamma\in L^1(\R^d)$.

 \end{lemma}

  \begin{proof} It is clear that for any $d\in\{1,2\}$ and $s<t$,
  \begin{align}
  \varphi_{t,R}(s,y) &  \leq \int_{\R^d} G_{t-s}(x-y) dx = t-s,  \label{ineq:241} \\
  \int_{\R^d}\varphi_{t,R}(s,y) dy &= \int_{B_R} dx \int_{\R^d}dy G_{t-s}(x-y) \leq \omega_d (t-s) R^d. \label{ineq:242}
  \end{align}
Then, we can write 
   \begin{align*} 
    R^{-d}  \int_{\R^{2d}} \varphi_{t_1,R}(s,y) \varphi_{t_2,R}(s,z) \gamma(y-z)   dydz & \leq (t_1-s)   R^{-d}  \int_{\R^{2d}}  \varphi_{t_2,R}(s,z) \gamma(y-z)   dydz  \quad\text{by \eqref{ineq:241}} \\
    &\leq (t_1-s) \big\| \gamma\big\|_{L^1(\R^d)} R^{-d} \int_{\R^d}  \varphi_{t_2,R}(s,z) dz \\
    &\leq \omega_d (t_1-s)(t_2-s)  \big\| \gamma\big\|_{L^1(\R^d)} .
 \end{align*} 
  This gives us the desired uniform boundedness. \qedhere

     \end{proof}

In the end of this section, we state two useful estimates from the paper \cite{BNZ20}.

\begin{lemma}\label{LEM51} When $d=2$, $G_{t}(x) =\frac{1}{2\pi} \big[ t^2 - |x|^2 \big]^{-1/2}     \mathbf{1}_{\{ |x|<t \}} $.  The following estimates hold:

{\rm (1)}  For $q\in(1/2,1)$ and $s<t$, then
\[
 \int_s^t dr  \big( G_{t-r}^{2q} \ast G_{r-s}^{2q} \big)^{1/q}(z) \lesssim (t-s)^{\frac{1}{q}-1} G_{t-s}^{2-\frac{1}{q}}(z),
\]
where the implicit constant only depends on $q$; see Lemma 3.3 in \cite{BNZ20}.

\smallskip

{\rm (2)} For any $p\in(0,1)$ and $q\in(1/2,1)$ such that $p + 2q \leq 3$, we have for $s<t$,
\[
\int_s^t G_{t-r}^{2q}\ast G_{r-s}^p(z) dr \lesssim (t-s)^{3-p-2q} \mathbf{1}_{\{ |z|< t-s  \}},
\]
where the implicit constant only depends on $p$ and $q$; see Lemma 3.4 in \cite{BNZ20}.

\end{lemma}

\section{Proof of Proposition \ref{prop1}} \label{SEC3}

The strict stationarity  follows from  two facts:
\begin{enumerate}

\item
For each $y\in \R^d$,  the random field $\{ u(t,x+y) : x\in \R^d\}$  coincides almost surely with  the random field $u$ driven by the shifted  noise $W_y$ given by
\[
W_y(\phi)= \int_{\R_+} \int_{\R^d} \phi(s,x-y) W(ds,dx), ~\phi\in\mathfrak{H}.
\]
\item The noise $W_y$ has the same distribution as $W$, which allows to conclude the proof of the stationarity property.
\end{enumerate}
We refer readers to  Lemma 7.1 in \cite{CKNP19} and footnote 1 in \cite{DNZ18} for similar arguments.  Now let us prove the ergodicity and in view of \cite[Lemma 7.2]{CKNP19}, it suffices to prove
\[
V_R(t) := \text{Var} \left( R^{-d} \int_{B_R} \prod_{j=1}^k g_j \big(  u(t,x+ \zeta^j) \big)    dx\right) \xrightarrow{R\to\infty} 0,
\]
for any fixed $\zeta^1, \dots,\zeta^k \in \R^d$ and $g_1, \dots, g_k\in C_b(\R)$ such that each $g_j$ vanishes at zero and has Lipschitz constant bounded by $1$.

 Using the Poincar\'e inequality  \eqref{Poincare}, we obtain
\begin{align}  \nonumber
V_R(t) & \le  R^{-2d} \int_{B_R^2} dx dy   \big\vert  \text{Cov} ( \mathcal{R}(x),  \mathcal{R}(y)) \big\vert   \hskip 15pt  \text{with}~ \mathcal{R}(x):= \prod_{j=1}^k g_j \big(  u(t,x+ \zeta^j) \big)  \\  \label{EQ2}
& \le R^{-2d} \int_{B_R^2} \int_0^t  \int_{\R^{2d}}    \| D_{s,z} \mathcal{R}(x) \| _2\| D_{s,z'} \mathcal{R}(y) \| _2 \,\gamma(z-z')dz dz' ds  dxdy .
\end{align} 
By the chain rule \eqref{cr},
\[
\big\vert D_{s,z} \mathcal{R}(x) \big\vert \leq \mathbf{1}_{(0,t)}(s)  \sum_{j_0=1}^k \left\vert \prod_{j=1, j\not = j_0} ^k g_j(u(t,x+ \zeta^j)) \right\vert   \big\vert D_{s,y} u(t,x+ \zeta^{j_0}) \big\vert,
\]
which implies, for any $s\in [0,t]$, 
\begin{equation} \label{EQ1}
\| D_{s,z} \mathcal{R}(x) \|_2   \le  \max_{1\le j \le k} \sup_{a\in \R} | g_j(a) | ^{k-1} \sum_{j_0=1}^k        \big\|  D_{s,y} u(t,x+ \zeta^{j_0}) \big\|_2  \lesssim  \sum_{j_0=1}^k  G_{t-s}(x-y +\zeta^{j_0}),
\end{equation}
 where in the second inequality we used Theorem \ref{THM}. Plugging  \eqref{EQ1} into \eqref{EQ2}, yields
\begin{align*}  
V_R(t)  \lesssim  R^{-2d}\sum_{j,\ell =1}^k   \int_{B_R^2} \int_0^t  \int_{\R^{2d}}    G_{t-s}(x-z +\zeta^{j})G_{t-s}(y-z' +\zeta^{\ell})\gamma(z-z')dz dz' ds  dxdy.
\end{align*} 
Using
\[
 \int_{B_R} dy  G_{t-s}(y-z' +\zeta^{\ell}) \leq  \int_{\R^d} dy  G_{t-s}(y) \lesssim t-s \quad\text{and}\quad \int_{\R^d} dz' \gamma(z-z') <\infty,
 \]
we deduce that  $V_R(t)  \lesssim  R^{-d}$. This finish the proof of Proposition \ref{prop1}. \hfill $\square$\\

\section{Proof of Theorem \ref{MAIN}} \label{SEC:Proof11}
The proof will be decomposed in several steps.
    
 \subsection{Asymptotic behavior of the covariance} \label{sub41}
 For any $t_1,t_2 \ge 0$,  in view of the stationarity of the random field $\{ (u(t_1,x),  u(t_2,x)), x\in \R^d\}$, we have
 \[
  \E\big[ F_R(t_1) F_R(t_2) \big]  =  \int_{B_R^2} \E [ (u(t,x-y)-1)(u(t,0)-1)] dx dy.
  \]
  By the dominated convergence theorem,  we obtain
  \begin{equation}  \label{R1}
  \lim_{R\rightarrow \infty}   R^{-d} \E\big[ F_R(t_1) F_R(t_2) \big]  = \omega_d \int_{\R^d} {\rm Cov} (u(t_1,x), u(t_2,0)) dx,
  \end{equation}
  provided $x\in\R^d \longmapsto \big\vert  {\rm Cov} (u(t_1,x), u(t_2,0)) \big\vert $ is integrable. In the next lemma we show this integrability property.
  \begin{lemma}  For any $t_1,t_2 \ge 0$,
  \[
    \int_{\R^d} \big\vert  {\rm Cov} (u(t_1,x), u(t_2,0)) \big\vert dx <\infty.
    \]
    \end{lemma}
    \begin{proof}
    Fix $t_1,t_2 \in [0,T]$. Using  Poincar\'e inequality \eqref{Poincare} and the estimate \eqref{IMP} yields
    \begin{align*}
   & \int_{\R^d}   dx |{\rm Cov} (u(t_1,x), u(t_2,0)) |  \\
   & \le     \int_{\R^d}dx \int_0^ { t_1\wedge t_2}  \int_{ \R^{2d}} dydzds  \| D_{s,y} u(t_1,x) \|_2  \| D_{s,z} u(t_2,0) \|_2
    \gamma(y-z)  dydzds  \\
   & \le       (C_{2,t,L,\gamma} \kappa_{2,t,L})^2   \int_{\R^d}dx \int_0^ { t_1\wedge t_2}  \int_{ \R^{2d}} ds  G_{t_1-s}(x-y) G_{t_2-s}(z)  \gamma(y-z) dydz,
    \end{align*} 
    which is   finite, by integrating successively in the variables $x$, $y$, $z$ and using the integrability of $\gamma$.
    This completes the proof.
    \end{proof}

  \subsection{Convergence of the finite-dimensional distributions}\label{sec42}
From the mild equation \eqref{mild} satisfied by $u(t,x)$ and using the fact that the Dalang-Walsh integral coincides with the divergence operator, we can write
 \begin{equation} \label{EQ4}
 F_R(t) = \int_{B_R} (u(t,x) -1)dx = \delta(V_{t,R}),
 \end{equation}
 with
 \begin{equation}
V_{t,R}(s,y) = \varphi_{t,R}(s,y) \sigma(u(s,y)),
\end{equation}
where $\varphi_{t,R}(s,y) $ has been defined in \eqref{varphi}.

The next proposition is the basic ingredient for the convergence  of the finite-dimensional distributions and also for the total variation bound  in \eqref{QCLT}.

\begin{proposition}\label{prop42}
For any $t_1, t_2\in [0,T]$,
 \begin{align}\label{finalbdd}
         {\rm Var}\big(   \langle DF_R(t_1), V_{t_2,R}  \rangle_\HH \big) \lesssim R^{d} ~\text{for $R\geq t_1+t_2$}.             
 \end{align}
\end{proposition} 
 
 Together with Proposition \ref{pr:TV}, the above estimate \eqref{finalbdd} leads us to the total variation bound in \eqref{QCLT}.
 
 \begin{proof}[Proof of Proposition  \ref{prop42}]
Note that
\[
D_{s,y}F_R(t) =   \varphi_{t,R}(s,y) \sigma(u(s,y)) +  \int_s^t\int_{\R^d} \varphi_{t,R}(r,z) \Sigma_{r,z} D_{s,y} u(r,z) W(dr, dz),
\]
where $\{ \Sigma_{r,z}: (r,z)\in\R_+\times\R^d \}$ is an adapted random field that is uniformly bounded by $L$, the Lipschitz constant of $\sigma$; see condition \textbf{(C1)}.
Then, for $t_1, t_2\in [0,T]$, we can write 
$\big\langle DF_R(t_1), V_{t_2, R} \big\rangle_\HH = A_1 + A_2$, with 
\begin{align*}
A_1 &= \big\langle  V_{t_1, R}, V_{t_2, R} \big\rangle_\HH =\int_0^{t_1\wedge t_2} ds  \int_{\R^{2d}} dy dz     \gamma(y-z) \varphi_{t_1, R}(s,y) \varphi_{t_2, R}(s,z) \sigma\big( u(s,y)\big) \sigma\big( u(s,z)\big)  
\end{align*}
and
\begin{align*}
A_2 &=\int_0 ^{t_1\wedge t_2} ds  \int_{\R^{2d}}  dydy'\gamma(y-y')  V_{t_2,R}(s,y') \left( \int_s^{t_1}\int_{\R^2} \varphi_{t_1, R}(r,z) \Sigma_{r,z}  D_{s,y}u(r,z)  W(dr,dz)\right). 
\end{align*}
It is clear that $\text{Var}\big(  \langle DF_R(t_1), V_{t_2, R} \rangle_\HH \big) \lesssim \text{Var}( A_1)  +  \text{Var}( A_2)$.  So in the sequel, we need to prove
\begin{center}
$\text{Var}( A_j)  \lesssim  R^d$   for $j=1,2$.
 \end{center}
Following the same strategy as in \cite[Section 4.2]{BNZ20}, we only need to prove 
\begin{align}\label{TU:suf}
\sup_{s\leq t_1\wedge t_2} \big(    \mathcal{T}_s +  \mathcal{U}_s\big) \lesssim  R^d,
\end{align}
where for
      $s\in(0, t_1\wedge t_2]$,
 \begin{align*} 
 \mathcal{T}_s &=   \int_0^s dr    \int_{\R^{6d}}  \varphi_{t_1,R}(s,y)     \varphi_{t_1,R}(s,y')    \varphi_{t_2,R}(s,z)   \varphi_{t_2,R}(s,z')   G_{s-r}(z-\xi)    G_{s-r}(z'-\xi')        \\
& \qquad\qquad\qquad  \times     \gamma( \xi   - \xi')    \gamma( y-z) \gamma( y'-z')   d\xi   d\xi'  dydz dy'dz' dx_1dx_1'dx_2dx_2'
 \end{align*}
  and
    \begin{align*} 
\mathcal{U}_s&=   \int_s^{t_1} dr  \int_{\R^{6d}} dzd\wt{z} dydy'd\wt{y} d\wt{y'}  \gamma( y -y') \gamma( \wt{y} -\wt{y'})   \gamma( z - \wt{z} )\\
&\quad\times \varphi_{t_2, R}(s, y')\varphi_{t_2, R}(s, \wt{y'}) \varphi_{t_1, R}(r,z)\varphi_{t_1, R}(r,\wt{z})  G_{r-s}(y-z)G_{r-s}(\wt{y}-\wt{z}).
\end{align*}
  In what follows, we only prove $\sup\big\{   \mathcal{T}_s: s\leq t_1\wedge t_2 \big\} \lesssim R^d$   and we omit the other part because $\mathcal{U}_s$ has the same-type expression as $\mathcal{T}_s$.

Using  \eqref{ineq:241},   we can  write
 \begin{align*} 
 \mathcal{T}_s &\leq (t_1-s) (t_2-s)   \int_0^s dr  \int_{\R^{6d}}  \varphi_{t_1,R}(s,y)   \varphi_{t_2,R}(s,z)  \gamma( y-z)    G_{s-r}(z-\xi)         \\
& \qquad  \times   G_{s-r}(z'-\xi')     \gamma( \xi   - \xi')  \gamma( y'-z')   d\xi   d\xi'  dydz dy'dz',
 \end{align*}
then we perform integration with respect to $dy', dz', d\xi', d\xi$ inductively to write 
 \begin{align*} 
 \mathcal{T}_s &\lesssim     \int_0^s dr  \int_{\R^{2d}}  \varphi_{t_1,R}(s,y)   \varphi_{t_2,R}(s,z)  \gamma( y-z)      dydz  \lesssim R^d,
 \end{align*}
where the last estimate follows from Lemma \ref{lemfact}. This leads to the bound \eqref{finalbdd}. \qedhere

 \end{proof}

 We are ready   to show the convergence of the finite-dimensional distributions.
 Let us choose  $m\ge 1$ points $t_1,\ldots,t_m\in(0\,,\infty)$.
	 Consider the  random vector $\Phi_R= \big(F_R(t_1), \dots, F_R(t_m) \big)$  and	  let $\mathbf{G}=(\mathcal{G}_1 \,,\ldots,\mathcal{G}_m)$ denote a centered Gaussian random vector
	with covariance matrix $(\mathcal{C}_{i,j}))_{1\le i,j\le m}$, where
	\[
\mathcal{C}_{i,j} :=	 \omega_d    \int_{\R^d} \text{\rm Cov} ( u(t_i,\xi),  u(t_j,0) )  d\xi.
	\]
	Recall from \eqref{EQ4} that  
	$F_R(t_i) =\delta( V_{t_i,R})$ for all $i=1,\ldots,m$. 	Then, by a  generalization of  a result of Nourdin, Peccati and  R\'eveillac (see \emph{e.g.} \cite[Theorem 6.1.2]{NP}), we can write
		\begin{equation} \label{equa7}
		\big\vert  \E( h(R^{-d/2}\Phi_R)) -\E (h(\mathbf{G})) \big\vert  \leq \frac{m}{2} \|h ''\|_\infty 
		\sqrt{   \sum_{i,j=1}^m   \E \left( \left|
		 \mathcal{C}_{i,j} - R^{-d} \langle  DF_R(t_i)  \,, V_{t_j, R} \rangle_{\HH} \right|^2
		\right)}
	\end{equation}
	for every $h\in C^2(\R^m)$ with bounded second partial derivatives, where
	 \[
		\|h'' \| _\infty= \max_{1\le i,j \le m}  \sup_{x\in\R^m}  \left| \frac { \partial ^2h (x) } {\partial x_i \partial x_j} \right|;
	\]
	see also Proposition 2.3 in \cite{HNV18}.   Thus,  in view of \eqref{equa7},  in order to show the  convergence in law of $R^{-d/2} \Phi_R$ to $\mathbf{G}$, it suffices to show that for any $i,j =1,\dots, m$, 
	\begin{equation} \label{h6}
	\lim_{R\rightarrow \infty}  \E \left( \left\vert   \mathcal{C}_{i,j} - R^{-d} \langle  DF_R(t_i)  \,, V_{t_j, R} \rangle_{\HH} \right\vert^2
		\right)=0.
		\end{equation}
	Notice that, by the duality relation  \eqref{D:delta} and the convergence \eqref{R1},   we have
	\begin{align} 
	R^{-d}\E \Big( \langle DF_R(t_i) \,,V_{t_j,R}   \rangle_{\HH} \Big) &=  R^{-d}	\E   \Big(   F_R(t_i)  \delta( V_{t_j,R} )  \Big)=  R^{-d} 	\E   \Big(   F_R(t_i)  F_R(t_j)   \Big)   \xrightarrow{R\to+\infty}  \mathcal{C}_{i,j}. \label{h7}
	\end{align}
	Therefore, the convergence \eqref{h6} follows immediately from  \eqref{h7} and  \eqref{finalbdd}.
		Hence the finite-dimensional distributions of
	$ \big\{R^{-d/2} F_R(t): t\in\R_+\big\}$ converge to those of
	$\mathcal{G}$ as $R\to\infty$.   \hfill $\square$\\

\subsection{Tightness via the criterion of   Chentsov-Kolmogorov.}\label{sec43}  In what follow, we appeal to the tightness criterion of Chentsov-Kolmogorov (see \emph{e.g.} \cite{Kunita86}) and we only need to obtain the following moment estimate: For any $p\geq 2$ and $0\leq s<t\leq T \leq R$,
\begin{align}\label{tight}
R^{-d/2} \big\| F_R(t) - F_R(s) \big\| _p \lesssim   (t-s)^{1/d},
\end{align}
where the implicit constant does not depend on $(R, s, t)$.

\begin{proof}[Proof of \eqref{tight}] Recall that
\[
F_R(t) =  \int_{\R_+\times\R^d} \varphi_{t,R}(r,y) \sigma(u_{r,y}) W(dr,dy).
\]
Then by BDG inequality  (see Lemma \ref{BDG}) and Minkowski's inequality,
\begin{align}
 & \big\| F_R(t) - F_R(s) \big\| _p^2  \lesssim  \Bigg\| \int_{\R_+\times\R^{2d}} \big( \varphi_{t,R}(r,y) - \varphi_{s,R}(r,y) \big)  \big( \varphi_{t,R}(r,y') - \varphi_{s,R}(r,y') \big)  \notag  \\
  &\qquad\qquad\qquad \qquad\qquad  \times   \sigma(u(r,y)) \sigma(u(r,y')) \gamma(y-y') dydy'dr \Bigg\|_{p/2}  \notag\\
  & \lesssim  \int_{\R_+\times\R^{2d}} \big| \varphi_{t,R}(r,y) - \varphi_{s,R}(r,y) \big|  \big| \varphi_{t,R}(r,y') - \varphi_{s,R}(r,y') \big|   \notag  \\
   &\qquad\qquad \qquad \times  \big\| \sigma(u(r,y)) \sigma(u(r,y')) \big\|_{p/2} \gamma(y-y') dydy'dr  \notag\\
 &\lesssim  \int_{\R_+\times\R^{2d}} \ \left( \big\vert \varphi_{t,R}(r,y) - \varphi_{s,R}(r,y) \big\vert^2 +   \big\vert \varphi_{t,R}(r,y') - \varphi_{s,R}(r,y') \big\vert^2\right)   \gamma(y-y') dydy'dr \label{s:app}\\
 &=\int_{\R_+\times\R^{2d}}  \big\vert \varphi_{t,R}(r,y) - \varphi_{s,R}(r,y) \big\vert^2        \gamma(y-y') dydy'dr, \notag
 \end{align} 
 where we have used  the following two facts to obtain \eqref{s:app}:
 \begin{itemize}
 \item[(i)] $\| u(r,y)\|_p$ is uniformly bounded on $[0,T]\times\R^d$, \quad (ii) $|ab|\leq \frac{1}{2}(a^2+b^2)$ for any $a,b\in\R$.

 \end{itemize}
 Integrating first with respect to $dy'$ yields, 
 \begin{align*}
 \big\| F_R(t) - F_R(s) \big\| _p^2   \lesssim  \int_{\R_+\times\R^{d}}  \big\vert \varphi_{t,R}(r,y) - \varphi_{s,R}(r,y) \big\vert^2        dydr.
 \end{align*}
 By direct computation (see also \cite[Section 4.3]{BNZ20} for the 2D case and \cite[Equation (4.2)]{DNZ18} for the 1D case),
 \[
  \big\vert \varphi_{t,R}(r,y) - \varphi_{s,R}(r,y) \big\vert^2  \lesssim (t-s)^{2/d}\1_{\{ | y| \leq R+t\}} \leq (t-s)^{2/d}\1_{\{ | y| \leq 2R\}}
 \]
 from which we have
  \begin{align*}
 \big\| F_R(t) - F_R(s) \big\| _p^2   \lesssim  \int_0^t  \int_{\R^{d}}  (t-s)^{2/d}\1_{\{ | y| \leq 2R\}}          dydr \lesssim R^d (t-s)^{2/d}.
 \end{align*}
 This gives us the desired tightness.  \end{proof}

 Combing the results from Sections \ref{sub41}, \ref{sec42} and \ref{sec43},  we can  complete the proof of Theorem \ref{MAIN}.

\medskip

\section{Proof of Theorem \ref{THM}} \label{SEC:Proof12}

 \subsection{Moment estimates for  Picard approximations}  \label{Picard}
 We define $u_0(t,x)=1$ and for $n\geq 0$,
 \[
 u_{n+1}(t,x) = 1 + \int_0^t \int_{\R^2} G_{t-s}(x-y) \sigma\big( u_n(s,y) \big) W(ds, dy).   
 \]
 It is a classic result that $u_n(t,x)$ converges in $L^p(\Omega)$ to $u(t,x)$ uniformly in $x\in\R^d$ for any $p\ge 2$; see \emph{e.g.} \cite[Theorem 4.3]{Dalang}.  
 If  $\sigma(1)=0$, we will end up in the trivial case where $u(t,x)\equiv 1$, in view of the above iteration, which explains the imposed condition \textbf{(C3)}.
 
  We will first derive moment estimates for $u_n(t,x)$.
  By  BDG and Minkowski's  inequalities, we can  write with $n\geq 1$,
   \begin{align*}
 &\| u_n(t,x)\|_p^2 \\
 & \leq  2+ 8p   \int_0^t  ds   \int_{\R^{2d}} dy dy' G_{t-s}(x-y) G_{t-s}(x-y')  \gamma(y-y')   \big\| \sigma ( u_{n-1}(s,y))  \sigma ( u_{n-1}(s,y') )\big\|_{p/2}     \\
 &\le  2+ 8p      \int_0^t  ds   \int_{\R^{2d}} dy dy' G_{t-s}(x-y) G_{t-s}(x-y') \gamma(y-y')    \big\| \sigma ( u_{n-1}(s,y))\big\|_p^2   
 \intertext{since  $ \| \sigma ( u_{n-1}(s,y))  \sigma ( u_{n-1}(s,y') )\|_{p/2}  \leq  \frac{1}{2}\big( \| \sigma ( u_{n-1}(s,y)) \|_{p}^2+   \| \sigma ( u_{n-1}(s,y')) \|_{p}^2\big)$; }
&\leq  2+ 8p      \int_0^t  ds   \int_{\R^{2d}} dy dy' G_{t-s}(x-y) G_{t-s}(x-y') \gamma(y-y')    \Big(  2 \sigma(0)^2+ 2L^2 \big\|  u_{n-1}(s,y) \big\|_p^2\Big).
 \end{align*}
Then, it follows from the estimate \eqref{FA1} that
 \begin{equation} \label{eq11}
 H_n(t) \le c_1 + c_2 \int_0^t  dsH_{n-1}(s),  
 \end{equation}
 where $H_n(t)= \sup_{x\in \R^d} \| u_n(t,x)\|_p^2$,  
 \[
 c_1:= 2+  16p\sigma(0)^2  t \mathfrak{m}_t
\quad  {\rm
  and
 } \quad
 c_2:=  16pL^2  \mathfrak{m}_t.
 \]
 Therefore, by iterating the inequality \eqref{eq11} and taking into account that $H_0(t)=1$, yields
 \[
 H_n(t) \le c_1 \exp(c_2 t).
 \]
That is, we obtain
 \[
  \| u_n(t,x)\|_p^2 \leq \big( 2+  16p\sigma(0)^2  t \mathfrak{m}_t \big) \exp\big( 16pL^2 t  \mathfrak{m}_t  \big).
  \]
  As a consequence,  
  \begin{equation}  \label{sigmap}
  \| \sigma(u_n(t,x)) \|_p \leq  \vert \sigma(0)\vert +  L  \Big( \sqrt{2}+ 4\sqrt{p} | \sigma(0)| \sqrt{ t \mathfrak{m}_t }     \Big) \exp\big( 8pL^2 t  \mathfrak{m}_t  \big)    =: \kappa_{p,t,L}.
  \end{equation}
 
  \subsection{Moment estimates for the  derivative of Picard approximations}
 Now, let us derive moment estimates for the derivative of the  Picard approximations. Our goal in this section is to establish that for $n\geq 4$,
  \begin{equation}  \label{Goalsec}
 \| D_{s,y} u_{n+1}(t,x) \| _p \le C_{p,t,L, \gamma} \kappa_{p,t,L} G_{t-s}(x-y),
 \end{equation}
 where   the constant  $ \kappa_{p,t,L}$ is defined in \eqref{sigmap} and the constant  $C_{p,t,L,\gamma}$ is given by \eqref{def:CPTL} in 1D case and by  \eqref{def:CPTL2} in 2D case.

 \begin{proof}[Proof of \eqref{Goalsec}]
 It is  known that 
for each $n\ge 0$, $u_n(t,x)\in \mathbb{D}^{1,p}$ with 
 \begin{align*}
 D_{s,y} u_{n+1}(t,x) = G_{t-s}(x-y)  \sigma\big( u_n(s,y) \big) + \int_s^t\int_{\R^d} G_{t-r}(x-z)  \Sigma_{r,z}^{(n)} D_{s,y}u_n(r,z) W(dr,dz),
 \end{align*}
 where  $\big\{\Sigma_{s,y}^{(n)}: (s,y)\in\R_+\times\R^d\big\}$ is an adapted random field that is uniformly bounded by $L$, for each $n$.  Now finite  iterations yield (with $r_0=t, z_0=x$)
  \begin{align}
   & D_{s,y} u_{n+1}(t,x)  = G_{t-s}(x-y)  \sigma\big( u_n(s,y) \big)\notag \\
   & \quad +  \int_s^t\int_{\R^d} G_{t-r_1}(x-z_1) \Sigma^{(n)}_{r_1,z_1} G_{r_1 - s}(z_1-y)  \sigma(  u_{n-1}(r_1,z_1) ) W(dr_1,dz_1) \notag \\
   &\qquad +   \sum_{k=2}^n   \int_s^t\cdots\int_s^{r_{k-1}} \int_{\R^{2k}}G_{r_k -s}(z_k-y) \sigma\big( u_{n-k}(r_k,z_k) \big) \notag \\
   &\qquad\qquad  \times \prod_{j=1}^{k} G_{r_{j-1} - r_{j}}(z_{j-1} - z_{j}) \Sigma^{(n+1-j)}_{r_j,z_j}  W(dr_j,dz_j)=:   \sum_{k=0}^n T^{(n)}_k,  \label{finiteit} 
  \end{align}   
 where $T^{(n)}_k$ denotes the $k$th item. For example,  $T_0^{(n)}=G_{t-s}(x-y)  \sigma\big( u_n(s,y) \big)$ and 
 $$
 T^{(n)}_1 = \int_s^t\int_{\R^d} G_{t-r_1}(x-z_1) \Sigma^{(n)}_{r_1,z_1} G_{r_1 - s}(z_1-y)  \sigma(  u_{n-1}(r_1,z_1) ) W(dr_1,dz_1).
 $$
We are going to estimate   $\| T^{(n)}_k \|_p$ for each $k=0, \dots, n$.
 
 \medskip
 \noindent
 {\it Case $k=0$}:  It is clear that
 \begin{equation} \label{k=0}
 \| T^{(n)}_0\|_p  \le  \kappa_{p,t,L} G_{t-s}(x-y),
 \end{equation}
 where $\kappa_{p,t,L}$ is the constant defined in \eqref{sigmap}.  
  
  \medskip
 \noindent
 {\it Case $k=1$}:  Applying  BDG and Minkowski's inequalities, we can write
 \begin{align} \notag
&\| T^{(n)}_1\|_p^2   \leq 4p \Bigg\|  \int_s^t dr_1 \int_{\R^{2d}} dz_1 dz'_1  G_{t-r_1}(x-z_1)G_{t-r_1}(x-z_1') G_{r_1 - s}(z_1-y) G_{r_1 - s}(z_1'-y)\\
&     \times    \label{k=1}
\gamma (z_1-z'_1)\Sigma^{(n)}_{r_1,z_1}  \Sigma^{(n)}_{r'_1,z'_1}  \sigma(  u_{n-1}(r_1,z_1) )  \sigma(  u_{n-1}(r_1,z'_1) )  \Bigg\|_{p/2}  \leq  4p L^2\kappa_{p,t,L}^2  \mathbf{K}_{s,t}(x,y),  
  \end{align}
  where
  \begin{align}
 \mathbf{K}_{s,t}(x,y)= \int_s^t dr  \int_{\R^{2d}} g_r(z)  (g_r\ast\gamma)(z) dz \label{eq:1w}
\end{align}
with the notation $g_r (z)= G_{t-r}(x-z) G_{r - s}(z-y) $.

\medskip
 \noindent
 {\it Case} $2\le k\le n$:
 We can write 
  \[
T^{(n)}_k =   \int_s^t\int_{\R^d}   G_{t-r_1}(x-z_1) \Sigma^{(n)}_{r_1,z_1} N_{r_1,z_1} W(dr_1, dz_1)
  \]
  with  
  \begin{align*}
  N_{r_1,z_1}&=
 \int_{ s<r_k < \cdots <r_2<r_1} \int_{\R^{(2k-2)d}}G_{r_k -s}(z_k-y) \sigma\big( u_{n-k}(r_k,z_k) \big)  \\
 & \qquad \times     \prod_{j=2}^{k} G_{r_{j-1} - r_{j}}(z_{j-1} - z_{j}) \Sigma^{(n+1-j)}_{r_j,z_j}  W(dr_j,dz_j),
  \end{align*}
which is clearly $\mathscr{F}_{r_1}$-measurable.     Then, by BDG   inequality,
we obtain
  \begin{align}
  &  \big\| T^{(n)}_k \big\|_p^2   \le 4p   \Bigg\|  \int_s^t dr_1      \int_{\R^{2d}} G_{t-r_1}(x-z_1) \Sigma^{(n)}_{r_1,z_1} N_{r_1,z_1}G_{t-r_1}(x-z'_1) \Sigma^{(n)}_{r_1,z'_1} N_{r_1,z'_1} \notag   \\
     &\qquad\qquad\qquad\qquad \times  \gamma( z_1' - z_1)  dz_1dz_1'     \Bigg\|_{p/2} \notag \\
 &\leq 4pL^2 \Bigg\|  \int_s^t dr_1      \int_{\R^{2d}} G_{t-r_1}(x-z_1)   N^2_{r_1,z_1}G_{t-r_1}(x-z'_1)      \gamma( z_1' - z_1)  dz_1dz_1'     \Bigg\|_{p/2}  \label{ineq:ab}\\
  &\leq 4pL^2  \int_s^t dr_1      \int_{\R^{2d}} G_{t-r_1}(x-z_1)  G_{t-r_1}(x-z'_1)    \| N_{r_1,z_1} \|_p^2   \gamma( z_1' - z_1)  dz_1dz_1'  ,   \notag
   \end{align}
where we used $|ab| \leq \frac{a^2+b^2}{2}$ in  the second inequality and we applied Minkowski's inequality in the last step.

  Now we can iterate the above process to obtain
\begin{align*}      
  \big\| T^{(n)}_k \big\|_p^2     &\leq (4pL^2)^{k-1} \int_s^t dr_1\int_s^{r_1}\cdots\int_{s}^{r_{k-2} } dr_{k-1} \int_{\R^{2dk-2d}} dz_1\cdots dz_{k-1} dz_1' ...dz_{k-1}' \\
  &\quad \times \left(\prod_{j=0}^{k-2}  G_{r_j - r_{j+1}}(z_j  -z_{j+1})  G_{r_j - r_{j+1}}(z_j  -z'_{j+1})  \gamma(z_{j+1} -z'_{j+1})    \right)    \big\| \wh{N}_{r_{k-1}, z_{k-1}} \big\|_p^2,
 \end{align*}      
          where $z_0=x, r_0=t$ and $\wh{N}_{r_{k-1}, z_{k-1}} $ is given by 
     \begin{align*}
 \widehat{N}_{r_{k-1},z_{k-1}} &:=  \int_{[ s, r_{k-1} ]\times\R^d}  W(dr_k, dz_k) \sigma\big( u_{n-k}(r_k,z_k)  \big)G_{r_{k-1}-r_k}(z_{k-1}-z_k) \\
 & \qquad \times \Sigma^{(n+1-k)}_{r_k,z_k} G_{r_k-s}(z_k-y).
   \end{align*}
By the same arguments that led to \eqref{k=1},  we have 
\[
     \big\| \wh{N}_{r_{k-1}, z_{k-1}} \big\|_p^2  \leq 4pL^2 \kappa_{p,t,L}^2  \mathbf{K}_{s, {r_{k-1}} }(x,y) ,
     \]
     which implies
     \begin{align}       \notag
  \big\| T^{(n)}_k \big\|_p^2     &\leq (4pL^2)^{k}  \kappa_{p,t,L}^2 \int_s^t dr_1\int_s^{r_1}\cdots\int_{s}^{r_{k-2} } dr_{k-1} \int_{\R^{2dk-2d}} dz_1\cdots dz_{k-1} dz_1' ...dz_{k-1}' \\
  &\quad \times \left(\prod_{j=0}^{k-2}  G_{r_j - r_{j+1}}(z_j  -z_{j+1})  G_{r_j - r_{j+1}}(z_j  -z'_{j+1})  \gamma(z_{j+1} -z'_{j+1})    \right)  \mathbf{K}_{s, {r_{k-1}} }(x,y).   \label{EQ15}
 \end{align}     
 
 To complete the estimation of the quantities    $\| T^{(n)}_k \|_p$ for $k=1, \dots,n$,  
  we consider separately the cases $d=1$ and $d=2$.
             
     \medskip
     
     \noindent
     {\bf Case $d=1$}:   In this case, $G_{t-r}(x-z) =\frac{1}{2}\1_{\{ |x-z |<t-r \}} $, so that,  using  the integrability of $\gamma$ and \eqref{G1}  yields
\begin{align}
 \mathbf{K}_{s,t}(x,y) &\leq  \frac{1}{4}\1_{\{ |x-y|<t-s\} } \| \gamma\|_{L^1(\R)}\int_s^t dr \int_{\R} dzG_{t-r}(x-z)  G_{r - s}(z-y)     \notag \\
 &\leq  \frac{1}{8}\1_{\{ |x-y|<t-s\} } \| \gamma\|_{L^1(\R)}(t-s)^2 G_{t-s}(x-y)    \notag \\
& \leq  \frac{t^2 \| \gamma\|_{L^1(\R)}  }{8} G_{t-s}(x-y)   . \label{ineq:d=1} 
\end{align}
Plugging this bound into \eqref{k=1} yields
 \begin{equation}  \label{ineq:k=d=1}
   \| T^{(n)}_1\|_p  \leq   tL \kappa_{p,t,L} \sqrt{p   \| \gamma\|_{L^1(\R)} }   G_{t-s}(x-y). 
  \end{equation}
  For $k=2,\dots, n$, from \eqref{EQ15} and \eqref{ineq:d=1}, we obtain
\begin{align}      
  \big\| T^{(n)}_k \big\|_p^2   &  \leq  \frac{ \| \gamma\|_{L^1(\R)}}{8} (4pL^2)^{k} t^2 \kappa_{p,t,L}^2   \int_s^t dr_1\cdots\int_{s}^{r_{k-2} } dr_{k-1} \notag \\
   & \quad \times \int_{\R^{2k-2}} dz_1...dz_{k-1} dz_1' ...dz_{k-1}'  G_{r_{k-1} - s}(z_{k-1}-y) \notag \\
  &  \quad \times \left(\prod_{j=0}^{k-2}  G_{r_j - r_{j+1}}(z_j  -z_{j+1})  G_{r_j - r_{j+1}}(z_j  -z'_{j+1})  \gamma(z_{j+1} -z'_{j+1})    \right).   \label{ineq:bdd=1}
 \end{align}  
 In the particular case $k=2$,   we obtain
\begin{align*}
  &\quad  \int_s^t dr \int_{\R^2}dz  dz' G_{t-r}(x-z)G_{t-r}(x-z')  \gamma(z-z')G_{r-s}(z-y) \\
   & \leq \frac{1}{4} \1_{\{ |x-y| < t-s\}}  \int_s^t dr \int_{\R^2}dz  dz'  G_{t-r}(x-z')  \gamma(z-z')   \leq   \frac{ \| \gamma\|_{L^1(\R)}}{4}G_{t-s}(x-y)    (t-s)^2,
\end{align*}
   which yields
 \begin{align}\label{bdd:k=2,d=1}  
    \big\| T^{(n)}_2\big\|_p   \leq      \Big(  t^2   \| \gamma\|_{L^1(\R)}  pL^2 \kappa_{p,t,L}   \Big)  G_{t-s}(x-y).
   \end{align}
  For $ 3\leq k\leq n$, we rewrite the spatial integral in   \eqref{ineq:bdd=1}  as
 \begin{align*}
&  \int_{\R^{2k-2}} dz_1...dz_{k-1} dz_1' \cdots dz_{k-1}'  \left(\prod_{j=0}^{k-2}  \1_{  \Big\{  \substack{ |z_j  -z_{j+1}| < r_j - r_{j+1}  \\   |z_j  -z'_{j+1}| < r_j - r_{j+1}    }  \Big\} }     \gamma(z_{j+1} -z'_{j+1})    \right) \frac{ \1_{\{      | z_{k-1}-y| < r_{k-1} - s \}   }  }{2^{2k-1}}  \\
 &\leq  \frac{ \1_{\{      | x-y| < t - s \}   }  }{2^{2k-1}} \int_{\R^{2k-4}} dz_2dz_2' \cdots dz_{k-1} dz'_{k-1}   \1_{  \big\{  \substack{ |x  -z_{2}| < t - r_2 \\   |x  -z'_2| < t - r_2    }  \big\} }
    \Bigg( \prod_{j=2}^{k-2}  \1_{  \big\{  \substack{ |z_j  -z_{j+1}| < r_j - r_{j+1}  \\   |z_j  -z'_{j+1}| < r_j - r_{j+1}    }  \big\} }  \\ 
     &\qquad \times \gamma(z_{j+1} -z'_{j+1})    \Bigg)   \gamma(z_2-z_2')   \int_{\R^2}dz_1dz_1'    \gamma(z_1 -z'_1 )\1_{   \{   |x  -z_{1}| < t - r_1        \} } 
   \end{align*}
   Note that ${\displaystyle
\int_{\R^2}dz_1dz_1'    \gamma(z_1 -z'_1 )\1_{   \{   |x  -z_{1}| < t - r_1        \} }  \leq 2t \| \gamma\|_{L^1(\R)} }$
  and then we can iterate the above process to deduce that the  spatial integral in   \eqref{ineq:bdd=1} can be bounded by 
  \[
   \frac{ \1_{\{      | x-y| < t - s \}   }  }{2^{2k-1}} \big(2t \| \gamma\|_{L^1(\R)}  \big)^{k-1} =  \frac{ \1_{\{      | x-y| < t - s \}   }  }{2^k}  \big(t \| \gamma\|_{L^1(\R)} \big)^{k-1}.
  \]
Thus, from  \eqref{ineq:bdd=1},
   \begin{align*}      
  \big\| T^{(n)}_k \big\|_p^2    
  &\leq \frac{1}{8} \kappa_{p,t,L}^2  \frac{ \big(2pL^2 t^2    \| \gamma\|_{L^1(\R)}  \big)^k }{(k-1)!}  \1_{\{      | x-y| < t - s \} }.
 \end{align*}   
 That is, 
 \begin{align} \label{bdd:d=1k>2}
  \big\| T^{(n)}_k \big\|_p \leq  \kappa_{p,t,L}  \frac{ \big(2pL^2 t^2    \| \gamma\|_{L^1(\R)}  \big)^{k/2} }{\sqrt{(k-1)!}} G_{t-s}(x-y).
  \end{align}
Now combining the estimates in \eqref{k=0},  \eqref{ineq:k=d=1}, \eqref{bdd:k=2,d=1} and \eqref{bdd:d=1k>2} yields 
\[
\big\| D_{s,y} u_{n+1}(t,x) \big\|_p \leq \sum_{k=0}^n  \big\| T^{(n)}_k \big\|_p \leq \kappa_{p,t,L} C_{p,t,L,\gamma} G_{t-s}(x-y),
\]
with 
\begin{align}\label{def:CPTL}
C_{p,t,L,\gamma} : = 1+\sum_{k=1}^\infty  \frac{ \big(2pL^2 t^2   \| \gamma\|_{L^1(\R)}  \big)^{k/2} }{\sqrt{(k-1)!}}.
\end{align}

\medskip

\noindent
{\bf Case $d=2$}:  Recall  $G_{t-r}(x-z) =\frac{1}{2\pi} \big[ (t-r)^2 - |x-z|^2 \big]^{-1/2}     \1_{\{ |x-z|<t-r \}} $ and 
 \begin{align}
 \mathbf{K}_{s,t}(x,y) = \int_s^t dr  \int_{\R^{2d}} g_r(z)  (g_r\ast\gamma)(z) dz 
\end{align}
with   $g_r (z)= G_{t-r}(x-z) G_{r - s}(z-y) $, see \eqref{eq:1w}.  By H\"older's inequality and Young's inequality, we obtain
\begin{align*}
\int_{\R^{2d}} g_r(z)  (g_r\ast\gamma)(z) dz  \le  \| g_r\|_{L^{2q}(\R^2)} \big\|g_r\ast\gamma \big\|_{L^{\frac{2q}{2q-1}}(\R^2)   } \leq  \| g_r\|_{L^{2q}(\R^2)}^2 \| \gamma\|_{L^{\pmb{\ell} }(\R^2)}, 
\end{align*}
where $q := \frac{\pmb{\ell}}{2\pmb{\ell}-1}\in (1/2, 1)$. Therefore, 
\begin{align} \notag
\mathbf{K}_{s,t}(x,y) & \leq  \| \gamma\|_{L^{\pmb{\ell}}(\R^2)} \int_s^t dr  \big( G_{t-r}^{2q} \ast G_{r-s}^{2q} \big)^{1/q}(x-y) \\
& \le  C_{\LL} (t-s)^{ (\LL-1)/\pmb{\ell}  }  \| \gamma\|_{L^{\LL}(\R^2)} G_{t-s}^{  1/\pmb{\ell}   }(x-y), \label{EQ16}
\end{align}
where the last inequality  follows from Lemma \ref{LEM51} and here and in the rest of the paper, $C_{\LL}$ will denote a generic constant that only depends on $\LL$ and may vary from line to line.

Then we deduce from \eqref{k=1} that
   \begin{align}  \label{ineq:k=1d=2}
 \| T^{(n)}_1\|_p  \leq    C_{\LL}  L \kappa_{p,t,L} \sqrt{p \| \gamma\|_{L^{\LL}(\R^2)} }  (t-s)^{  \frac{\pmb{\ell}-1}{2 \pmb{\ell}}}   G^{\frac{1}{2\pmb{\ell}}}_{t-s}(x-y). 
  \end{align} 
  Note that $G^{\frac{1}{2\pmb{\ell}}}_{t-s}(x-y) \leq   (2\pi)^{ 1-\frac 1{2\LL}} (t-s)^{1 - \frac{1}{2\LL}} G_{t-s}(x-y)$. Therefore, from \eqref{ineq:k=1d=2} we can write 
  \begin{align} \label{casek=1}
   \| T^{(n)}_1\|_p  \leq   C_{\LL}  L \kappa_{p,t,L} \sqrt{p \| \gamma\|_{L^{\LL}(\R^2)} }  t^{  \frac{3\pmb{\ell}-2}{2 \pmb{\ell}}}    G_{t-s}(x-y).
\end{align}

   Consider now the case $k\in\{2,\dots,n\}$. We have, from \eqref{EQ15} and  \eqref{EQ16}
  \begin{align}      
 \big\| T^{(n)}_k \big\|_p^2   &  \leq   (4pL^2)^{k}  \kappa_{p,t,L}^2  (t-s)^{\frac{\pmb{\ell}-1}{\pmb{\ell}}} \| \gamma\|_{L^{\pmb{\ell}}(\R^2)}  \int_s^t dr_1\cdots\int_{s}^{r_{k-2} } dr_{k-1} \notag \\
 &\quad \times \int_{\R^{4k-4}} dz_1...dz_{k-1} dz_1' ...dz_{k-1}'  G^{1/\pmb{\ell}}_{r_{k-1} - s}(z_{k-1}-y) \notag \\
 &\quad \times \left(\prod_{j=0}^{k-2}  G_{r_j - r_{j+1}}(z_j  -z_{j+1})  G_{r_j - r_{j+1}}(z_j  -z'_{j+1})  \gamma(z_{j+1} -z'_{j+1})    \right) . \label{ineq:bdd=2}
 \end{align}
For $k=2$, we deduce from \eqref{ineq:bdd=2} 
\begin{align}\label{ineq:k=d=2}
 \big\| T^{(n)}_2 \big\|_p^2 & \leq   (4pL^2)^{2}  \kappa_{p,t,L}^2  (t-s)^{\frac{\pmb{\ell}-1}{\pmb{\ell}}} \| \gamma\|_{L^{\pmb{\ell}}(\R^2)} \wh{ \mathbf{K}}_{s,t}(x-y),
\end{align}  
with 
\begin{align}\label{def:Khat}
\wh{ \mathbf{K}}_{s,t}(x-y):=  \int_s^t dr \int_{\R^4}dzdz'   G^{1/\pmb{\ell}}_{r - s}(z-y)      G_{t -r}(x -z)  G_{t -  r }(x  -z')  \gamma(z -z').   
\end{align}
We can write, with $h_r(z):= G_{t-r}(x-z)G_{r-s}^{1/\LL}(z-y)$ and $q=\frac{\pmb{\ell}}{2\pmb{\ell}-1}$ ,
\begin{align*}
\wh{ \mathbf{K}}_{s,t}(x-y) &=  \int_s^t dr \int_{\R^2}dz'     G_{t -  r }(x  -z')  (\gamma\ast h_r)(z') \\
&\leq    \int_s^t dr \big\| G_{t-r}\big\|_{L^{2q}(\R^2)} \big\| \gamma \big\|_{L^{\LL}(\R^2)} \big\| h_r \big\|_{L^{2q}(\R^2)},
\end{align*}
where the last inequality follows from H\"older's inequality and Young's convolution inequality.

By direct computation, $\big\| G_{t-r}\big\|_{L^{2q}(\R^2)}  = \left(  \frac{(2\pi)^{1-2q}  }{2-2q} \right)^{\frac{1}{2q}} (t-r)^{\frac{1-q}{q}}$. Then,  
\begin{align*}
 \wh{ \mathbf{K}}_{s,t}(x-y) & \leq   \left(  \frac{(2\pi)^{1-2q}  }{2-2q} \right)^{\frac{1}{2q}} \big\| \gamma \big\|_{L^{\LL}(\R^2)}  t^{\frac{1-q}{q}}   \int_s^t dr  \sqrt{   \big( G_{t-r}^{2q} \ast G_{r-s}^{2q/\LL}  \big)^{1/q}(x-y)   }   \\
 &\leq  \left(  \frac{(2\pi)^{1-2q}  }{2-2q} \right)^{\frac{1}{2q}} \big\| \gamma \big\|_{L^{\LL}(\R^2)}  t^{\frac{2-q}{2q}}    \left( \int_s^t dr   \big( G_{t-r}^{2q} \ast G_{r-s}^{2q/\LL}  \big)^{1/q}(x-y)   \right)^{\frac{1}{2}}, 
\end{align*}
where we used the Jensen's inequality for finite measure in the last estimate.  Using $G_{r-s}^{2q/\LL}(z-y) \leq (2\pi)^{\frac{2\LL-2}{2\LL-1} } r^{\frac{\LL-1}{2\LL-1}} G_{r-s}^{2q}(z-y)$ and applying Lemma \ref{LEM51}, we  obtain 
\begin{align*}
 \int_s^t dr   \big( G_{t-r}^{2q} \ast G_{r-s}^{2q/\LL}  \big)^{1/q}(x-y)  &  \leq (2\pi)^{\frac{2\LL-2}{\LL}} t^{\frac{\LL-1}{\LL}} \int_s^t dr   \big( G_{t-r}^{2q} \ast G_{r-s}^{2q}  \big)^{1/q}(x-y) \\
 &\leq C_{\LL} t^{\frac{2\LL-2}{\LL}} G^{1/\LL}_{t-s}(x-y).
\end{align*}
 Therefore,
\begin{align}
 \wh{ \mathbf{K}}_{s,t}(x-y)   \leq  C_{\LL}  \big\| \gamma \big\|_{L^{\LL}(\R^2)}  t^{\frac{2-q}{2q}}    \left( t^{\frac{2\LL-2}{\LL}} G^{1/\LL}_{t-s}(x-y)   \right)^{\frac{1}{2}} = C_{\LL}  \big\| \gamma \big\|_{L^{\LL}(\R^2)}   t^{\frac{5\LL-4}{2\LL}} G^{\frac{1}{2\LL}}_{t-s}(x-y)  \label{EST:Khat}
\end{align}
from which, together with \eqref{ineq:k=d=2}, we obtain
\begin{align}
\big\| T_2^{(n)} \big\|_p &\leq C_{\LL}   \left(   t^{\frac{7\LL-6}{2\LL}} \big(4pL^2 \kappa_{p,t,L}  \| \gamma  \|_{L^{\LL}(\R^2)} \big)^2 G^{\frac{1}{2\LL}}_{t-s}(x-y) \right)^{1/2} \notag \\
&\leq C_{\LL}     t^{\frac{7 \LL-6}{4\LL}}  4pL^2 \kappa_{p,t,L}  \| \gamma  \|_{L^{\LL}(\R^2)}   G_{t-s}(x-y), \label{QQ1}
\end{align}
where we used $G^{\frac{1}{4\LL}}_{t-s}(x-y)  \leq 2\pi t^{1- \frac{1}{4\LL}} G_{t-s}(x-y)$ to obtain the last estimate.  
 
\medskip

For $k\in\{3,\dots ,n\}$, we first point out that  the following integral
\begin{align*}
&\int_s^{r_{k-2}}dr_{k-1} \int_{\R^4} dz_{k-1} dz_{k-1}'  G^{1/\pmb{\ell}}_{r_{k-1} - s}(z_{k-1}-y)  G_{r_{k-2} - r_{k-1}}(z_{k-2} -z_{k-1})  \\
&\qquad \times G_{r_{k-2} - r_{k-1}}(z_{k-2} -z'_{k-1})       \gamma(z_{k-1} -z'_{k-1})  
\end{align*}
is exactly $\wh{\mathbf{K}}_{s, r_{k-2}}(z_{k-2} -y)$, see \eqref{def:Khat}. This is bounded by $ C_{\LL}  \big\| \gamma \big\|_{L^{\LL}(\R^2)}   t^{\frac{5\LL-4}{2\LL}} G^{\frac{1}{2\LL}}_{r_{k-2}-s}(z_{k-2}-y) $, in view of $r_{k-2} \leq t$ and \eqref{EST:Khat}.

 Then, we have 
  \begin{align*}       
 \big\| T^{(n)}_k \big\|_p^2  & \leq  C_{\LL} (4pL^2)^{k}  \kappa_{p,t,L}^2  t^{\frac{7\LL-6}{2\LL}  } \| \gamma\|_{L^{\pmb{\ell}}(\R^2)}^2 \int_s^t dr_1\cdots\int_{s}^{r_{k-3} } dr_{k-2} \int_{\R^{4k-8}} dz_1...dz_{k-2} dz_1' ...dz_{k-2}'   \\
 & \quad  \times \left( \prod_{j=0}^{k-3}  G_{r_j - r_{j+1}}(z_j  -z_{j+1})  G_{r_j - r_{j+1}}(z_j  -z'_{j+1})  \gamma(z_{j+1} -z'_{j+1})  \right)  G_{r_{k-2}-s}^{\frac{1}{2\LL}} (z_{k-2}-y).
 \end{align*}
Similar to the estimation of $\wh{\mathbf{K}}_{s,t}$, we write with  $\wt{h}_r(z):= G_{t-r}(x-z)G_{r-s}^{\frac{1}{2\LL}}(z-y)$ and $q =\frac{\LL}{2\LL-1}$,
\begin{align*}
\wt{\mathbf{K}}_{s,t}(x-y) :&= \int_s^t dr \int_{\R^2} G_{t-r}(x-z') \big(\gamma\ast  \wt{h}_r  \big)(z') \leq    \int_s^t dr \big\| G_{t-r}\big\|_{L^{2q}(\R^2)} \big\| \gamma \big\|_{L^{\LL}(\R^2)} \big\| \wt{h}_r \big\|_{L^{2q}(\R^2)} \\
&\leq  C_{\LL} t^{\frac{1-q}{q}} \| \gamma\|_{L^{\LL}(\R^2)} \int_s^t dr \left(  \int_{\R^2}G^{2q}_{t-r}(x-z)G_{r-s}^{q/\LL}(z-y) dz \right)^{\frac{1}{2q}}.
\end{align*}
Since $2q + \frac{q}{\LL} = \frac{2\LL+1}{2\LL-1} < 3$, we can apply Lemma \ref{LEM51} to write
\[
 \left(  \int_s^t dr  \int_{\R^2}G^{2q}_{t-r}(x-z)G_{r-s}^{q/\LL}(z-y) dz \right)^{\frac{1}{2q}} \leq C_{\LL} t^{\frac{2\LL-2}{\LL}} \mathbf{1}_{\{ |x-y| < t-s \}}.
 \]
Thus,
\begin{align}\label{EST:Ktilde}
\wt{\mathbf{K}}_{s,t}(x-y) \leq     C_{\LL} t^{\frac{3\LL-3}{\LL}}  \| \gamma\|_{L^{\LL}(\R^2)}   \mathbf{1}_{\{ |x-y| < t-s \}}.
\end{align}
From this estimate, we deduce 
  \begin{align}       
 \big\| T^{(n)}_3 \big\|_p  & \leq  C_{\LL}  \sqrt{ (4pL^2)^{3}  \kappa_{p,t,L}^2  t^{\frac{7\LL-6}{2\LL}  } \| \gamma\|_{L^{\pmb{\ell}}(\R^2)}^2 \wt{\mathbf{K}}_{s,t}(x-y) } \notag  \\
 &\leq C_{\LL} \big( 4pL^2   \| \gamma\|_{L^{\pmb{\ell}}(\R^2)} \big)^{3/2} \kappa_{p,t,L}  t^{\frac{13\LL-12}{4\LL}  } \mathbf{1}_{\{ |x-y| < t-s \}} \notag \\
 &\leq C_{\LL} \big(4 pL^2   \| \gamma\|_{L^{\pmb{\ell}}(\R^2)} \big)^{3/2} \kappa_{p,t,L}  t^{\frac{17\LL-12}{4\LL}  }  G_{t-s}(x-y), \label{QQ2}
 \end{align}
where we also used the fact $\mathbf{1}_{\{ |x-y| < t-s \}} \leq 2\pi t G_{t-s}(x-y)$.

\smallskip

For $4\leq k\leq n$, we write
  \begin{align*}       
 &\big\| T^{(n)}_k \big\|_p^2   \leq   C_{\LL}(4pL^2)^{k}  \kappa_{p,t,L}^2  t^{\frac{7\LL-6}{2\LL}  } \| \gamma\|_{L^{\pmb{\ell}}(\R^2)}^2 \int_s^t dr_1\cdots\int_{s}^{r_{k-4} } dr_{k-3} \int_{\R^{4k-12}} dz_1...dz_{k-3} dz_1' ...dz_{k-3}'   \\
 & \quad  \times \left( \prod_{j=0}^{k-4}  G_{r_j - r_{j+1}}(z_j  -z_{j+1})  G_{r_j - r_{j+1}}(z_j  -z'_{j+1})  \gamma(z_{j+1} -z'_{j+1})  \right)  \wt{\mathbf{K}}_{s, r_{k-3}}(z_{k-3}-y)  \\
  &\leq  C_{\LL} (4pL^2)^{k}  \kappa_{p,t,L}^2  t^{\frac{13\LL-12}{2\LL}  } \| \gamma\|_{L^{\pmb{\ell}}(\R^2)}^3 \1_{\{ |x-y| <t-s\}} \int_s^t dr_1\cdots\int_{s}^{r_{k-4} } dr_{k-3}    \\
 &    \times\int_{\R^{4k-12}} dz_1 \cdots dz_{k-3} dz_1'  \cdots dz_{k-3}' \left( \prod_{j=0}^{k-4}  G_{r_j - r_{j+1}}(z_j  -z_{j+1})  G_{r_j - r_{j+1}}(z_j  -z'_{j+1})  \gamma(z_{j+1} -z'_{j+1})  \right)  
  \end{align*}
  using \eqref{EST:Ktilde}. Now we can perform integration  inductively with respect to $dz_{k-3}dz'_{k-3}$, \dots , $dz_1dz_1'$ and we get
  \begin{align*}
 & \int_{\R^{4k-12}} dz_1 \cdots dz_{k-3} dz_1' ...dz_{k-3}'  \prod_{j=0}^{k-4}  G_{r_j - r_{j+1}}(z_j  -z_{j+1})  G_{r_j - r_{j+1}}(z_j  -z'_{j+1})  \gamma(z_{j+1} -z'_{j+1})     \leq \mathfrak{m}_t^{k-3}
  \end{align*}
  so that
    \begin{align}       
 \big\| T^{(n)}_k \big\|_p    &\leq  C_{\LL} \sqrt{ (4pL^2)^{k}  \kappa_{p,t,L}^2  t^{\frac{13\LL-12}{2\LL}  } \| \gamma\|_{L^{\pmb{\ell}}(\R^2)}^3 \1_{\{ |x-y| <t-s\}}  \frac{ (t \mathfrak{m}_t)^{k-3}}{(k-3)!} } \notag  \\
 &\leq C_{\LL} (4pL^2)^{k/2} \kappa_{p,t,L}   t^{\frac{17\LL-12}{4\LL}  } \| \gamma\|_{L^{\pmb{\ell}}(\R^2)}^{3/2} \sqrt{ \frac{ (t \mathfrak{m}_t)^{k-3}}{(k-3)!} } G_{t-s}(x-y). \label{QQ3}
  \end{align} 
  Combining \eqref{k=0}, \eqref{casek=1}, \eqref{QQ1}, \eqref{QQ2} and  \eqref{QQ3} yields
 \begin{equation}  \label{Conq1}
 \| D_{s,y} u_{n+1}(t,x) \| _p \le C_{p,t,L, \gamma} \kappa_{p,t,L} G_{t-s}(x-y),
 \end{equation}
 with
\begin{align} 
C_{p,t,L,\gamma} : &= 1 +  C_{\LL} L t^{\frac{3\LL-2}{2\LL}} \sqrt{p \| \gamma \|_{L^{\LL}(\R^2)}}  + C_{\LL} pL^2   t^{\frac{7\LL-6}{4\LL}} \| \gamma\|_{L^{\LL}(\R^2)} \notag  \\
 &\qquad \qquad + C_{\LL}  t^{\frac{17\LL-12}{4\LL}}  \big(pL^2 \| \gamma \|_{L^{\LL}(\R^2) } \big)^{3/2}     \sum_{k=0}^\infty     \frac{ \big(4pL^2  t\mathfrak{m}_t \big)^{k/2}}{\sqrt{k!}}. \label{def:CPTL2}
\end{align}
 This concludes the proof of \eqref{Goalsec}.
 \end{proof}
 
\subsection{Proof of  Theorem \ref{THM}}
We can now proceed with the proof of  Theorem \ref{THM}. We first apply  Minkowski's inequality  and  \eqref{Goalsec}  to obtain
\begin{align*}
\big\| Du_{n+1}(t,x) \big\|^2_{L^p(\Omega; \HH)} & \leq \int_0^t dr \int_{\R^{2d}} dzdy \gamma(z-y) \big\| D_{s,y}u_{n+1}(t,x) D_{s,z}u_{n+1}(t,x)  \big\|_{p/2} \\
&\lesssim  \int_0^t dr \int_{\R^{2d}} dzdy\gamma(z-y) \big\| D_{s,y}u_{n+1}(t,x)\big\|_p \big\| D_{s,z}u_{n+1}(t,x)  \big\|_{p} \\
&\lesssim \int_0^t dr \int_{\R^{2d}} dzdy\gamma(z-y)  G_{t-s}(x-y) G_{t-s}(x-z),
\end{align*}
which is uniformly bounded.  Then standard Malliavin calculus arguments imply that  up to a subsequence $Du_{n_k}(t,x)$ converges to $Du_{t,x}$ with respect to the weak topology on $L^p(\Omega; \HH)$; see \emph{e.g.} \cite{MilletMarta}.  Similarly, for any $q\in(1,2)$,
\begin{align*}
&\big\| Du_{n+1}(t,x) \big\|^p_{L^p(\Omega; L^q(\R_+\times\R^d) )}  = \left\|  \int_{\R_+\times\R^{2d}} dsdy  \big\vert D_{s,y}u_{n+1}(t,x) \big\vert^{q}  \right\|_{p/q}^{p/q} \\
&\lesssim   \left( \int_{\R_+\times\R^{2d}} dsdy  \big\| D_{s,y}u_{n+1}(t,x) \big\|_p^{q}  \right)^{p/q}  \lesssim   \left( \int_{\R_+\times\R^{2d}} dsdy  G_{t-s}^q(x-y)  \right)^{p/q} \lesssim1.
\end{align*}
So $\big\{Du_{n_k}(t,x)\big\}$ has a further subsequence that converges to the same limit $Du_{t,x}$ with respect to the weak topology on $L^p\big(\Omega; L^q(\R_+\times\R^d) \big)$ and as a result, for $1< q < 2 \leq p<\infty$ and for any finite $T$,
\begin{align*}
\sup_{(t,x)\in[0,T]\times\R^d} \left\| \int_{\R_+\times\R^d} \big\vert  D_{s,y} u_{t,x} \big\vert^{q} dy ds \right\|_{p/q} <\infty.
\end{align*}
Therefore, following exactly the same lines in the proof of \cite[Theorem 1.2]{BNZ20} (step 4 therein), we can get the upper bound in \eqref{IMP}. The lower bound is straightforward in light of the formula of Clark-Ocone (Lemma \ref{CO}).  \hfill $\square$

 \bigskip
 
 \noindent\textbf{Acknowledgment:}  D. Nualart is supported by NSF Grant DMS 1811181.

 \end{document}